\documentclass[12pt]{article}
\usepackage{amsmath,amssymb,amsthm,bm}
\usepackage{graphicx}
\usepackage{txfonts}

\newtheorem{theorem}{Theorem}[section]
\newtheorem{proposition}[theorem]{Proposition}

\newtheorem{lemma}[theorem]{Lemma}
\newtheorem{remark}[theorem]{Remark}

\numberwithin{equation}{section}

\begin{document}

\begin{center}
\bfseries\Large
Partial Chebyshev Polynomials and Fan Graphs
\end{center}

\begin{center}
Wojciech M{\l}otkowski \\
Instytut Matematyczny \\
Uniwersytet Wroc{\l}awski \\
Plac Grunwaldzki 2, Wroc{\l}aw, 50-384 Poland \\
wojciech.mlotkowski@math.uni.wroc.pl
\\
\medskip
and
\\
\medskip
Nobuaki Obata \\
Center for Data-driven Science and Artificial Intelligence \\
Tohoku University \\
Sendai, 980-8576 Japan \\
obata@tohoku.ac.jp
\end{center}

\bigskip
{\bfseries Abstract} \enspace
Motivated by the product formula of the Chebyshev polynomials 
of the second kind $U_n(x)$, we newly introduce 
the partial Chebyshev polynomials
$U^{\mathrm{e}}_n(x)$ and $U^{\mathrm{o}}_n(x)$
and derive their basic properties,
relations to the classical Chebyshev polynomials,
and new factorization formulas for $U_n(x)$.
In order to calculate the quadratic embedding constant
(QEC) of a fan graph $K_1+P_n$,
we derive a new polynomial $\phi_n(x)$ which is factorized 
by partial Chebyshev polynomial $U^{\mathrm{e}}_n(x)$.
We prove that $\mathrm{QEC}(K_1+P_n)$ is given in terms of the
minimal zero of $\phi_n(x)$,
and obtain the explicit value of $\mathrm{QEC}(K_1+P_n)$ for
an even $n$ and its reasonable exstimate for an odd $n$.



\section{Introduction}

The Chebyshev polynomials, known as classical ortho\--gonal polynomials,
are widely used in various fields such as
differential and integral equations, approximation theory, 
asymptotic theory, mathematical statistics, and so on, 
see e.g., \cite{Mason-Handscomb2003, Rivlin1974}.
On the other hand, the algebraic graph theory has a long history 
and recently, much attention has been drawn not only to
the adjacency and combinatorial Laplace matrices
but also to the distance matrix
\cite{Bapat2010, Biggs1993, Brouwer-Haemers2012}.
A first connection between the Chebyshev polynomials and
algebraic graph theory is likely found in the fact that
the characteristic polynomial 
of the adjacency matrix $A_n$ of the path $P_n$ is given by
the Chebyshev polynomial of the second kind 
in such a way that $\det(xI-A_n)=U_n(x/2)=\Tilde{U}_n(x)$,
see e.g., \cite{Biggs1993, Cvetkovic-Rowlinson-Simic2010}.
From this observation
we may expect and forsee many fruitful developments.
Indeed, in chemical graph theory the Chebyshev polynomials
have been efficiently used for the study of characteristic polynomials
and matching polynomials of some special graphs
\cite{Balasubramanian2023,Balasubramanian2025,Hosoya-Randic1983}.
Moreover, the distance matrix provides important characteristics of
graphs such as distance level patterns, Wiener indices,
and distance spectra \cite{Mihalic-etal1992,Fowler-Caporossi-Hansen2001}.

Motivated by the aforementioned background and the embedding problem 
in Euclidean distance geometry
\cite{Alfakih2018,Balaji-Bapat2007, Deza-Laurent1997, Jaklic-Modic2013, 
Liberti-Lavor-Maculan-Mucherino2014}
tracing back to Blumenthal \cite{Blumenthal1953} and
Schoenberg \cite{Schoenberg1935,Schoenberg1938},
we arrived at the introduction of the quadratic embedding constant (QEC)
of a finite connected graph \cite{Obata2017,Obata-Zakiyyah2018}.
In recent years, research on the QEC has progressed for various classes of 
graphs \cite{Choudhury-Nandi2023, Irawan-Sugeng2021, Mlotkowski2022,
Obata2023b, Lou-Obata-Huang2022,MO-2018},
and interest in classification based on the QEC has been increasing
\cite{Baskoro-Obata2021, Baskoro-Obata2024,MSW2024,Obata2023a}.

In this paper, 
using the newly introduced ``partial Chebyshev polynomials,''
we solve the problem of evaluating the QEC of a fun graph $K_1+P_n$, 
that is, the graph join of the singleton graph $K_1$ and the path $P_n$ with $n\ge1$.
The fan graphs, as one of the basic families of graphs,
have been studied from various aspects,
for example, 
matching polynomials \cite{Balasubramanian2025},
coloring and chromatic polynomials
\cite{Falcon-Venkatachalam-Gowri-Nandini2021,
Maulana-Wijaya-Santoso2018,Roy2017},
the distance matrices \cite{Hao-Li-Zhang2022},
graph spectrum \cite{Liu-Yuan-Das2020},
subtree problem \cite{Yang-Wang-etal2019}, and so forth.
The calculation of $\mathrm{QEC}(K_1+P_n)$ is
not straightforward, contrary to the seemingly simple structure of a fan graph.
However, this can also be inferred from the heavy work on
$\mathrm{QEC}(P_n)$ in \cite{Mlotkowski2022}.
In the previous paper \cite{Mlotkowski-Obata2025},
after rather tedius computation based on differential calculus,
we evaluated $\mathrm{QEC}(K_1+P_n)$ and 
conjectured that $\mathrm{QEC}(K_1+P_n)$ forms a strictly
increasing sequence converging to $0$.
In this paper, by using the partial Chebyshev polynomials,
not only is the evaluation of $\mathrm{QEC}(K_1+P_n)$ greatly simplified, 
but the conjecture is also positively resolved.
Thus, our results form the foundation for the problem of
classifying graphs along the increasing sequence
$-1/2=\mathrm{QEC}(K_1+P_3)<\mathrm{QEC}(K_1+P_4)<\dots
\rightarrow0$,
while an attempt to classify graphs
along the increasing sequence 
$-1=\mathrm{QEC}(P_2)<\mathrm{QEC}(P_3)<\dots
\rightarrow-1/2$ has already begun \cite{Baskoro-Obata2021,Baskoro-Obata2024}. 

The paper is organized as follows.
In Section \ref{sec:Partial Chebyshev Polynomials} we define
the partial Chebyshev polynomials and show their basic properties
including some new factorization formulas.
In Section \ref{03sec:Quadratic Embedding Constants of Fan Graphs}
we formulate the main problem of calculating 
$\mathrm{QEC}(K_1+P_n)$,
derive preliminary formulas
using the new polynomials $\phi_n(x)$
and state the main results
(Theorems \ref{04thm:main formula for fan 1}
and \ref{04thm:main formula for fan 2}).
In Section \ref{sec:New Polynomials} we derive 
a factorization of $\phi_n(x)$ 
by means of partial Chebyshev polynomials
and we complete the proof of our main results
by determining the minimal zeroes of $\phi_n(x)$.

\section{Partial Chebyshev Polynomials}
\label{sec:Partial Chebyshev Polynomials}

\subsection{Motivation}

Following the standard notation,
for $n\ge0$ the \textit{Chebyshev polynomial of the second kind}
(of order $n$) is defined by
\begin{equation}\label{01eqn:def of U_n(x)}
U_n(x)=\frac{\sin(n+1)\theta}{\sin\theta},
\qquad x=\cos\theta.
\end{equation}
The point of our departure is the 
product formula of $U_n(x)$ given by
\begin{equation}\label{01eqn:factorization of Un(x)}
U_n(x)=\prod_{k=1}^n 2\left(x-\cos\frac{k\pi}{n+1}\right),
\qquad
n\ge1,
\end{equation}
which follows by examining the zeroes of $U_n(x)$
in the definition \eqref{01eqn:def of U_n(x)}.
As we will show in Theorems \ref{02thm:U=UeUo}
and \ref{02thm:product expressions},
this product can be naturally split into ``even part'' and ``odd part.''
In the following sections,
for clear comparison with classical Chebyshev polynomials
we define the partial Chebyshev polynomials
as quotients of trigonometric functions in a similar spirit
as the classical ones.
We then derive the factorization
\[
U_n(x)=U^{\mathrm{e}}_n(x)\cdot U^{\mathrm{o}}_n(x)
\]
and other basic properties.

\subsection{Definition and Recurrence Relations}
\label{02subsec:Partial Chebyshev Polynomials}

With each $n\ge0$ we associate new polynomials
$U^{\mathrm{o}}_n(x)$ and $U^{\mathrm{e}}_n(x)$ defined by
\begin{align}
U^{\mathrm{e}}_n(x)
&=\begin{cases}
\displaystyle\frac{\sin(n+1)\theta/2}{\sin\theta/2}
&\text{if $n$ is even,}\\[8pt]
\displaystyle\frac{\sin (n+1)\theta/2}{\sin\theta}
&\text{if $n$ is odd,}
\end{cases}
\qquad x=\cos\theta,
\label{02eqn:def of Ue} \\[4pt]
U^{\mathrm{o}}_n(x)
&=\begin{cases}
\displaystyle\frac{\cos(n+1)\theta/2}{\cos\theta/2}
&\text{if $n$ is even,}\\[8pt]
\displaystyle 2\cos\frac{(n+1)\theta}{2}
&\text{if $n$ is odd,}
\end{cases}
\qquad x=\cos\theta.
\label{02eqn:def of Uo}
\end{align}
In fact, by elementary trigonometric formulas and induction 
argument we easily see that the right-hand sides of 
\eqref{02eqn:def of Ue} and \eqref{02eqn:def of Uo} are
polynomials of $\cos\theta$.
Moreover, the following results are obtained
directly from \eqref{01eqn:def of U_n(x)},
\eqref{02eqn:def of Ue} and \eqref{02eqn:def of Uo}
with elementary trigonometric formulas.

\begin{theorem}\label{02thm:U=UeUo}
For $n\ge0$ we have the factorization
$U_n(x)=U^{\mathrm{e}}_n(x)\cdot U^{\mathrm{o}}_n(x)$.
\end{theorem}

\begin{theorem}\label{02thm:Ue and Uo by U}
For $n\ge0$ we have
\begin{align}
U^{\mathrm{e}}_{2n}(x)&=U_{n}(x)+U_{n-1}(x),
&U^{o}_{2n}(x)&=U_n(x)-U_{n-1}(x),
\qquad
\label{for:ueuoeven}\\
U^{\mathrm{e}}_{2n+1}(x)&=U_n(x),
&U^{\mathrm{o}}_{2n+1}(x)&=U_{n+1}(x)-U_{n-1}(x),
\label{for:ueuoodd}
\end{align}
where we tacitly understand that $U_{-1}(x)=0$.
\end{theorem}

\begin{remark}\normalfont
It is well known that the Chebyshev polynomials $U_n(x)$ fulfill
the three-term recurrence relation:
\begin{equation}\label{02eqn:recurrence of Un(x)}
P_{n+2}(x)-2x P_{n+1}(x)+P_{n}(x)=0,
\qquad n\ge0.
\end{equation}
We see that the four series of polynomials
$U^{\mathrm{e}}_{2n}(x), U^{\mathrm{o}}_{2n}(x), 
U^{\mathrm{e}}_{2n+1}(x)$ and $U^{\mathrm{o}}_{2n+1}(x)$
satisfy the same three-term recurrence relation as in
\eqref{02eqn:recurrence of Un(x)} with different initial conditions.
Since any polynomials $P_n(x)$ satisfying
\eqref{02eqn:recurrence of Un(x)} is expressible
as a linear combination of $U_n(x)$ and $U_{n-1}(x)$, 
by matching the coefficients (possibly polynomials)
we may obtain Theorem \ref{02thm:Ue and Uo by U}.
\end{remark}

In some contexts, the ``compressed'' Chebyshev polynomials
of the second kind defined by $\Tilde{U}_n(x)=U_n(x/2)$ are useful.
It is well known that $\Tilde{U}_n(x)=x^n+\dotsb$ is a monic polynomial
with integer coefficients.
In this connection we also note the following result,
of which the proof is immediate by Theorem \ref{02thm:Ue and Uo by U}.

\begin{proposition}
For $n\ge0$, the compressed partial Chebyshev polynomials
$\Tilde{U}^{\mathrm{e}}_n(x)=U^{\mathrm{e}}_n(x/2)$
and $\Tilde{U}^{\mathrm{o}}_n(x)=U^{\mathrm{o}}_n(x/2)$ are
monic polynomials with integer coefficients.
\end{proposition}

\begin{remark}\normalfont
Following \cite{Mason-Handscomb2003, Rivlin1974} let
$T_n(x)$, $V_n(x)$ and $W_n(x)$ be the Chebyshev polynomials 
of the first, third and fourth kinds.
It is shown by comparing the definitions that
$U^{\mathrm{e}}_{2n}(x)=W_n(x)$,
$U^{\mathrm{o}}_{2n}(x)=V_n(x)$,
$U^{\mathrm{e}}_{2n+1}(x)=U_{n}(x)$ and
$U^{\mathrm{o}}_{2n+1}(x)=2T_{n+1}(x)$ for $n\ge0$.
\end{remark}

\subsection{Factorization}

\begin{theorem}\label{02thm:product expressions}
For $n\ge0$ we have 
\begin{equation}\label{02eqn:factorization of Ue and Uo} \\
U^{\mathrm{e}}_n(x)
=\prod_{\substack{1\le k\le n \\ \text{$k$: even}}}
 2\left(x-\cos\frac{k\pi}{n+1}\right), 
\qquad
U^{\mathrm{o}}_n(x)
=\prod_{\substack{1\le k\le n \\ \text{$k$: odd}}}
 2\left(x-\cos\frac{k\pi}{n+1}\right),
\end{equation}
where we tacitly understand that 
$U^{\mathrm{e}}_0(x)=U^{\mathrm{e}}_1(x)=U^{\mathrm{o}}_0(x)=1$
(the range of $k$ in the product is empty).
\end{theorem}

\begin{proof}
We start with the ``even'' part.
Suppose first that $n\ge2$ is even.
By definition we have
\[
U^{\mathrm{e}}_n\left(\cos\frac{2l\pi}{n+1}\right)
=\frac{\sin l\pi}{\sin l\pi/(n+1)}
=0,
\qquad 1\le l\le n/2.
\]
Since $U^{\mathrm{e}}_n(x)$ is a polynomial of order $n/2$, 
we see that $\cos 2l\pi/(n+1)$ with $1\le l\le n/2$ 
exhaust the zeroes of $U^{\mathrm{e}}_n(x)$ 
and they are all simple.
The case of an odd $n\ge3$ is examined in a similar manner.
As a result, for $n\ge2$ the zeroes of $U^{\mathrm{e}}_n(x)$
are all simple and given by
$\cos 2l\pi/(n+1)$ with $1\le l \le [n/2]$.
Therefore, $U^{\mathrm{e}}_n(x)$ is a constant multiple of
\[
\prod_{\substack{1\le k\le n \\ \text{$k$: even}}}
 \left(x-\cos\frac{k\pi}{n+1}\right).
\]
It follows from Theorem \ref{02thm:Ue and Uo by U} that
the leading term of $U^{\mathrm{e}}_n(x)$ coincides with the one of
$U_m(x)=(2x)^m+\dotsb$,
where $m=n/2$ for an even $n$ and $m=(n-1)/2$ for an odd $n$.
Then, matching the coefficients of the leading terms,
we obtain the first relation 
in \eqref{02eqn:factorization of Ue and Uo}.
The second one follows in a similar manner.
\end{proof}

We note that
Theorem \ref{02thm:U=UeUo} is reproduced from
Theorem \ref{02thm:product expressions} with 
the product formula \eqref{01eqn:factorization of Un(x)}.
Moreover, we obtain factorization formulas for
the Chebyshev polynomials $U_n(x)$, which are not found in
literatures so far as we know.
While, other types of factorizations are known, see e.g., 
\cite{Keri2022,Rayes2005,Wolfram2022}.

From now on, we often write $U_n$ for $U_n(x)$.

\begin{theorem}\label{02thm:factorization of U}
For $n\ge0$ we have
\begin{align}
U_{2n} &=(U_{n}+U_{n-1})(U_n-U_{n-1}),
&U_{2n+1} &=U_n(U_{n+1}-U_{n-1}),
\label{02eqn:factorization of U_2n and U_2n+1} \\
U_{2n}-1&=U_{n-1}(U_{n+1}-U_{n-1}),
&U_{2n+1}-1&=(U_{n+1}-U_{n})(U_{n}+U_{n-1}),
\label{for:unminus1}\\
U_{2n}+1&=U_{n}(U_{n}-U_{n-2}),
&U_{2n+1}+1&=(U_{n+1}+U_{n})(U_{n}-U_{n-1}),
\label{for:unplus1}
\end{align}
where we tacitly understand 
that $U_{-1}=U_{-1}(x)=0$ and $U_{-2}=U_{-2}(x)=-1$.
\end{theorem}

\begin{proof}
We obtain \eqref{02eqn:factorization of U_2n and U_2n+1} 
by combining Theorems \ref{02thm:U=UeUo} and \ref{02thm:Ue and Uo by U}.
Then, \eqref{for:unminus1} and \eqref{for:unplus1} follow from 
\eqref{02eqn:factorization of U_2n and U_2n+1}
with the help of the well-known relation
$U_{n}^2-U_{n+1}U_{n-1}=1$.
\end{proof}

Below we record some more factorization formulas
based on the partial Chebyshev polynomials.
The proofs are computational and omitted.

\begin{theorem}\label{02thm:factorization of U (2)}
For $n\ge0$ we have
\begin{align}
U_{2n}-U_{2n-1}-1&=2(x-1)U_{n-1}(U_{n}+U_{n-1}),\\
U_{2n}+U_{2n-1}-1&=2(x+1)U_{n-1}(U_{n}-U_{n-1}),\\
U_{2n}-U_{2n-1}+1&=(U_{n}-U_{n-1})(U_{n}-U_{n-2}),\\
U_{2n}+U_{2n-1}+1&=(U_{n}+U_{n-1})(U_{n}-U_{n-2}),
\end{align}
and
\begin{align}
U_{2n+1}-U_{2n}-1&=2(x-1)U_{n}(U_{n}+U_{n-1}),\\
U_{2n+1}+U_{2n}+1&=2(x+1)U_{n}(U_{n}-U_{n-1}),\\
U_{2n+1}-U_{2n}+1&=(U_{n}-U_{n-1})(U_{n+1}-U_{n-1}),\\
U_{2n+1}+U_{2n}-1&=(U_{n}+U_{n-1})(U_{n+1}-U_{n-1}).
\end{align}
\end{theorem}


\section{Quadratic Embedding Constants of Fan Graphs}
\label{03sec:Quadratic Embedding Constants of Fan Graphs}

\subsection{Quadratic Embedding Constants of Graphs}

Let $G=(V,E)$ be a finite connected graph on $n=|V|\ge2$ vertices.
For two distinct vertices $i,j\in V$ 
the length of a shortest walk connecting them 
is denoted by $d(i,j)$ and is called
the \textit{graph distance} between $i$ and $j$.
We set $d(i,i)=0$ for any $i\in V$.
The \textit{distance matrix} of $G$ is defined by
$D=[d(i,j)]_{i,j\in V}$.
Following \cite{Obata2017,Obata-Zakiyyah2018}
the \textit{quadratic embedding constant} 
(\textit{QEC} for short) of $G$ is defined by
\begin{equation}\label{03eqn:def of QEC(G)}
\mathrm{QEC}(G)
=\max\{\langle f,Df \rangle\,;\, f\in C(V), \,
\langle f,f \rangle=1, \, \langle {\bf 1},f \rangle=0\},
\end{equation}
where $C(V)\cong \mathbb{R}^n$ is
the space of real column vectors $f=[f(i)]_{i\in V}$,
${\bf 1}\in C(V)$ the column vector whose entries are all one,
and $\langle\cdot,\cdot\rangle$ the canonical inner product.

The QEC provides a criterion for 
a finite connected graph $G=(V,E)$ to admit
a \textit{quadratic embedding} in a Euclidean space,
i.e., a map $\pi:V\rightarrow \mathbb{R}^N$ such that
$\|\pi(i)-\pi(j)\|^2= d(i,j)$ for any $i,j\in V$,
where the left-hand side stands for the square of the Euclidean 
distance.
Such a graph is called \textit{of QE class}.
It follows from the classical result of 
Schoenberg \cite{Schoenberg1935,Schoenberg1938}
that $G$ is of QE class if and only if 
the distance matrix $D$ is conditionally negative definite,
and hence if and only if $\mathrm{QEC}(G)\le 0$.
Beyond this motivating point there is a growing interest in 
exploring how effective the QEC works
in the classification of graphs.

\subsection{Fan Graphs}

For $n\ge1$ let $P_n$ be the path graph on $n$ vertices.
The \textit{fan graph} on $n+1$ vertices
is the graph join of the singleton graph $K_1$
and the path $P_n$, denoted by $K_1+P_n$,
namely, a graph on the vertex set
$V=\{0,1,2,\dots,n\}$ with the edge set
$E=\{\{i,i+1\}\,;\, 1\le i\le n-1\}\cup
\{\{0,i\}\,;\, 1\le i\le n\}$, see Figure \ref{fig:Fan graph}.
The distance matrix of $K_1+P_n$ is written down 
in a block matrix form:
\begin{equation}\label{03eqn:distance matrix of join}
D=\begin{bmatrix}
  0 & J \\
  J & 2J-2I-A_n
\end{bmatrix},
\qquad
A_n=
\begin{bmatrix}
0 & 1 &  \\
1 & 0 & 1 & \\
  & 1 & 0 & 1 & \\
  &         & \ddots  & \ddots & \ddots \\
  &         &         & 1       & 0 & 1 \\
  &         &         &        & 1 & 0
\end{bmatrix},
\end{equation}
where $J$ is the all-one matrix (the size is understood in the context),
$I=I_n$ the $n\times n$ identity matrix
and $A_n$ the adjacency matrix of $P_n$.

\begin{figure}[!h]
\centering\includegraphics[width=1.8in]{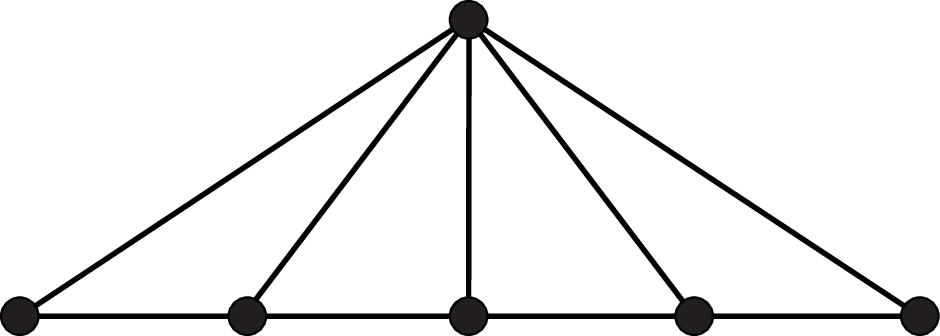}
\caption{Fan graph $K_1+P_n$ with $n=5$}
\label{fig:Fan graph}
\end{figure}

\subsection{Preliminary Formula}
\label{03subsec:Preliminary Formula}

For $\mathrm{QEC}(K_1+P_n)$ we need to find the conditional maximum
of the quadratic function:
\begin{align}
\psi(f,g)
=\left\langle \begin{bmatrix} f \\ g \end{bmatrix},
D \begin{bmatrix} f \\ g \end{bmatrix} \right\rangle 
&=2\langle f, Jg\rangle+\langle g, (2J-2I-A_n)g\rangle
\nonumber \\
&=2f\langle {\bf 1}, g\rangle
  +2\langle {\bf 1}, g\rangle^2- 2\langle g, g\rangle
  -\langle g, A_ng\rangle,
\label{03eqn:target quadratic function}
\end{align}
under the conditions
$f^2+\langle g,g\rangle=1$ and $f+\langle{\bf 1},g\rangle=0$,
where $f\in\mathbb{R}$ and $g\in\mathbb{R}^n$,
see \eqref{03eqn:def of QEC(G)}
and \eqref{03eqn:distance matrix of join}.
Upon applying the method of Lagrange's multipliers,
we set 
\begin{equation}\label{03eqn:varphi}
\varphi(f,g,\lambda,\mu)
=\psi(f,g)-\lambda(f^2+\langle g,g\rangle-1)
 -\mu(f+\langle {\bf 1},g\rangle).
\end{equation}
Let $\mathcal{S}$ be the set of stationary points
$(f,g,\lambda,\mu)\in \mathbb{R}\times
\mathbb{R}^n\times\mathbb{R}\times\mathbb{R}$ of \eqref{03eqn:varphi},
that is, the solutions to
\begin{equation}\label{03eqn:stationary points}
\frac{\partial\varphi}{\partial f}
=\frac{\partial\varphi}{\partial g_1}
=\dotsb
=\frac{\partial\varphi}{\partial g_n}
=\frac{\partial\varphi}{\partial \lambda}
=\frac{\partial\varphi}{\partial \mu}
=0.
\end{equation}
By general theory, the conditional maximum is 
attained at some stationary point.
On the other hand, we see easily 
that $\psi(f,g)=\lambda$ for any $(f,g,\lambda,\mu)\in\mathcal{S}$.
Thus, we come to the useful formula:
\begin{equation}\label{04eqn:QEC formula in general}
\mathrm{QEC}(K_1+P_n)
=\max \lambda(\mathcal{S}),
\end{equation}
where $\lambda(\mathcal{S})$ is the set of $\lambda\in\mathbb{R}$
appearing in the stationary points $\mathcal{S}$,
see \cite{Obata-Zakiyyah2018} for more details.

Calculating the derivatives 
in \eqref{03eqn:stationary points} explicitly
and switching $\lambda$ to a new variable $\alpha=-\lambda-2$,
the equations \eqref{03eqn:stationary points} are transferred
into the following equations:
\begin{align}
(\alpha+1) f &=\frac{\mu}{2},
\label{04eqn:eqn-Sa(1)} \\
(A_n-J-\alpha I)g &=-\frac{\mu}{2}\,{\bf 1},
\label{04eqn:eqn-Sa(2)} \\
f^2+\langle g,g\rangle &=1,
\label{04eqn:eqn-Sa(3)} \\
f+\langle {\bf 1},g\rangle &=0.
\label{04eqn:eqn-Sa(4)}
\end{align}
Let $\tilde{\alpha}_n$ be the minimal $\alpha\in\mathbb{R}$ appearing
in the solutions $(f,g,\alpha,\mu)
\in\mathbb{R}\times\mathbb{R}^n\times \mathbb{R}\times\mathbb{R}$
to the system of 
equations \eqref{04eqn:eqn-Sa(1)}--\eqref{04eqn:eqn-Sa(4)}.
Then, $\max\lambda(\mathcal{S})=-\tilde{\alpha}_n-2$ and 
\eqref{04eqn:QEC formula in general} becomes
\[
\mathrm{QEC}(K_1+P_n)=-\Tilde{\alpha}_n-2,
\qquad n\ge1.
\]
For $n=1,2$ the value is already known 
without having to recalculate it.
In fact, $K_1+P_1=K_2$ and $K_1+P_2=K_3$ are
the complete graphs of which the QEC are $-1$ \cite{Obata-Zakiyyah2018}.
Therefore, we have
$\mathrm{QEC}(K_1+P_1)=\mathrm{QEC}(K_1+P_2)=-1$ and
$\Tilde{\alpha}_1=\Tilde{\alpha}_2=-1$.

Hereafter we focus on the case of $n\ge3$.
It is known \cite{Baskoro-Obata2021} that
$\mathrm{QEC}(G)\ge -1$ for any graph $G$
and equality holds only for a complete graph $G=K_n$ with $n\ge2$.
Therefore, for $n\ge3$ we have $\mathrm{QEC}(K_1+P_n)>-1$
and $\tilde{\alpha}_n<-1$.
Thus, to obtain $\mathrm{QEC}(K_1+P_n)$ we need only to
find the solutions to
equations \eqref{04eqn:eqn-Sa(1)}--\eqref{04eqn:eqn-Sa(4)}
with $\alpha<-1$.
Under this condition, from \eqref{04eqn:eqn-Sa(1)} we obtain
\begin{equation}\label{04eqn:f=}
f=\frac{1}{\alpha+1}\,\frac{\mu}{2}\,.
\end{equation}
Since $Jg=\langle {\bf 1}, g\rangle{\bf 1}=-f{\bf 1}$ by
\eqref{04eqn:eqn-Sa(4)}, with the help of
\eqref{04eqn:f=} we see that \eqref{04eqn:eqn-Sa(2)} becomes
\begin{equation}\label{04eqn:(A-a)g=}
(A_n-\alpha I)g=-\frac{\alpha+2}{\alpha+1}\,\frac{\mu}{2}\,{\bf 1}.
\end{equation}
Thus, our task is to solve 
the system of equations \eqref{04eqn:f=}, \eqref{04eqn:(A-a)g=},
\eqref{04eqn:eqn-Sa(3)} and \eqref{04eqn:eqn-Sa(4)},
and to determine the minimal $\alpha\in(-\infty,-1)$
appearing in the solutions, which is $\Tilde{\alpha}_n$.
We solve this task by dividing it into two cases.
Let $\mathrm{ev}(A_n)$ denote the set of eigenvalues of $A_n$.
\par\noindent
\begin{description}
\item[(Case I)] Find the solutions with
$\alpha\in(-\infty,-1)\backslash \mathrm{ev}(A_n)$
and determine the minimal $\alpha$ appearing therein,
which is denoted by $\sigma_n$.
\item[(Case II)] Find the solutions with
$\alpha\in(-\infty,-1)\cap \mathrm{ev}(A_n)$
and determine the minimal $\alpha$ appearing therein,
which is denoted by $\tau_n$.
\end{description}

Summing up, we arrive at the preliminary formula for
$\mathrm{QEC}(K_1+P_n)$.

\begin{proposition}\label{03prop:preliminary formula}
For $n\ge3$ let $\Tilde{\alpha}_n$ be the minimal $\alpha\in\mathbb{R}$
appearing in the solutions to
the system of equations \eqref{04eqn:f=}, \eqref{04eqn:(A-a)g=},
\eqref{04eqn:eqn-Sa(3)} and \eqref{04eqn:eqn-Sa(4)}.
Then we have
\begin{equation}\label{04eqn:Tilde alpha by sigma and tau}
\tilde{\alpha}_n=\min\{\sigma_n,\tau_n\},
\qquad n\ge3,
\end{equation}
where $\sigma_n$ and $\tau_n$ are defined in (Case I) and
(Case II), respectively.
Moreover, we have
\begin{equation}\label{03eqn:main formula for QEC(K1+Pn)}
\mathrm{QEC}(K_1+P_n)=-\Tilde{\alpha}_n-2,
\qquad n\ge1,
\end{equation}
where $\Tilde{\alpha}_1=\Tilde{\alpha}_2=-1$.
\end{proposition}

The eigenvalues of $A_n$ are well known
\cite{Bapat2010,Biggs1993,Brouwer-Haemers2012,
Cvetkovic-Rowlinson-Simic2010}.
In fact, from the relation
$\det(xI-A_n)=U_n(x/2)=\Tilde{U}_n(x)$ for $n\ge1$
we obtain the eigenvalues of $A_n$ as 
the zeroes of $\Tilde{U}_n(x)$ as follows:
\begin{equation}\label{04eqn:ev(A_n)}
\mathrm{ev}(A_n)
=\left\{\omega^{(n)}_k\,;\, 1\le k\le n \right\},
\qquad
\omega^{(n)}_k=2\cos\frac{k\pi}{n+1},
\qquad n\ge1.
\end{equation}
Note that $\omega^{(n)}_k$ is the $k$th largest eigenvalue of $A_n$.

\subsection{Case I: Calculating $\sigma_n$}

We will find the solutions $(f,g,\alpha,\mu)
\in \mathbb{R}\times\mathbb{R}^n\times\mathbb{R}\times\mathbb{R}$
to equations
\eqref{04eqn:f=}, \eqref{04eqn:(A-a)g=},
\eqref{04eqn:eqn-Sa(3)} and \eqref{04eqn:eqn-Sa(4)}
with $\alpha\in(-\infty,-1)\backslash \mathrm{ev}(A_n)$.
Noting that $A_n-\alpha I$ can be inverted 
in \eqref{04eqn:(A-a)g=} and 
that $\mu=0$ does not appear in the solutions,
by simple algebra we see that \eqref{04eqn:eqn-Sa(4)} becomes
\begin{equation}\label{04eqn:(A-alpha)(2)}
(\alpha+2)\langle{\bf 1},(A_n-\alpha I)^{-1}{\bf 1}\rangle-1=0.
\end{equation}

\begin{lemma}\label{04lem:Key identity}
Let $n\ge1$.
For $\alpha\not\in \mathrm{ev}(A_n)\cup\{\pm2\}$ we have
\begin{equation}\label{04eqn:Key identity}
\langle{\bf 1},(A_n-\alpha I)^{-1}{\bf 1}\rangle
=\frac{1}{(2-\alpha)^2}
 \left\{n(2-\alpha)+2
 -2\left(\frac{\Tilde{U}_{n-1}(\alpha)}{\Tilde{U}_n(\alpha)}
 +\frac{1}{\Tilde{U}_n(\alpha)}\right)
\right\}.
\end{equation}
\end{lemma}

\begin{proof}
We set $g=(A_n-\alpha I)^{-1}{\bf 1}$ and $g=[g_1, \dots, g_n]^T$.
Then $g$ satisfies $(A_n-\alpha I)g={\bf 1}$,
which is equivalent to the three-term recurrence relation: 
\begin{equation}\label{04eqn:main 3-term recurrence relation}
g_{k+2}-\alpha g_{k+1}+g_k=1,
\qquad 0\le k\le n-1,
\end{equation}
with boundary condition $g_0=g_{n+1}=0$.
The unique solution is explicitly given by 
\[
g_k=\frac{1}{2-\alpha}\bigg(
 1-\frac{\xi^k}{1+\xi^{n+1}}
  -\frac{\eta^k}{1+\eta^{n+1}}\bigg),
\qquad
0\le k\le n+1,
\]
where $\xi,\eta\in\mathbb{C}$ be the characteristic roots of
\eqref{04eqn:main 3-term recurrence relation},
satisfying $\xi+\eta=\alpha$ and $\xi\eta=1$. 
Noting that $\xi\neq\eta$ by the assumption $\alpha\neq\pm2$,
by direct computation we come to
\begin{align}
\langle{\bf 1},(A_n-\alpha I)^{-1}{\bf 1}\rangle
&=\langle{\bf 1},g\rangle
=\sum_{k=1}^n g_k
\nonumber \\
&=\frac{1}{(2-\alpha)^2}
 \left\{n(2-\alpha)+2
 -2\left(\frac{\xi^n-\eta^n}{\xi^{n+1}-\eta^{n+1}}
 +\frac{\xi-\eta}{\xi^{n+1}-\eta^{n+1}}\right)
\right\}.
\label{04eqn:in proof Prop4.1(1)}
\end{align}
On the other hand, 
by comparing the recurrence relations we obtain
\begin{equation}\label{04eqn:relation to U_n}
\frac{\xi^{n+1}-\eta^{n+1}}{\xi-\eta}
=\Tilde{U}_n(\alpha),
\qquad n\ge0.
\end{equation}
Then, \eqref{04eqn:Key identity} follows from
\eqref{04eqn:in proof Prop4.1(1)} and \eqref{04eqn:relation to U_n}.
\end{proof}

With the help of Lemma \ref{04lem:Key identity} and by simple algebra
we see that \eqref{04eqn:(A-alpha)(2)} becomes
\begin{equation}\label{04eqn:Phi_n(alpha)=0}
((n+1)\alpha^2-6\alpha-4n)\Tilde{U}_n(\alpha)
+2(\alpha+2)(\Tilde{U}_{n-1}(\alpha)+1)
=0.
\end{equation}
Setting $\alpha=2x$ in the left-hand side,
we define the key polynomial $\phi_n(x)$ by
\begin{equation}\label{03eqn:def of phi(x)}
\phi_n(x)
=((n+1)x^2-3x-n)U_n(x)+(x+1)(U_{n-1}(x)+1).
\end{equation}
Obviously, $\phi_n(x)$ is a polynomial of
order $n+2$ and has integer coefficients.

\begin{lemma}\label{03lem:preparatory formula for sigma_n}
For $n\ge3$ we have
\begin{equation}\label{04eqn:sigma_n}
\sigma_n
=\min\{\alpha\in(-\infty,-1)\backslash \mathrm{ev}(A_n)\,;\,
\phi_n(\alpha/2)=0\}.
\end{equation}
\end{lemma}

\begin{proof}
Any $\alpha\in(-\infty,-1)\backslash \mathrm{ev}(A_n)$ satisfying
\eqref{04eqn:Phi_n(alpha)=0} gives rise to a solution to
the equations \eqref{04eqn:f=}, \eqref{04eqn:(A-a)g=},
\eqref{04eqn:eqn-Sa(3)} and \eqref{04eqn:eqn-Sa(4)}.
In fact, $f$ and $g$ are 
respectively determined by 
\eqref{04eqn:f=} and \eqref{04eqn:(A-a)g=},
and then $\mu$ is specified by \eqref{04eqn:eqn-Sa(3)}.
Therefore, $\sigma_n$ is the minimal 
$\alpha\in(-\infty,-1)\backslash \mathrm{ev}(A_n)$ satisfying
\eqref{04eqn:Phi_n(alpha)=0}.
Then \eqref{04eqn:sigma_n} is immediate 
by observing that
\eqref{04eqn:Phi_n(alpha)=0} is equivalent to
$\phi_n(\alpha/2)=0$ by \eqref{03eqn:def of phi(x)}.
\end{proof}

For the evaluation of \eqref{04eqn:sigma_n} we need
the factorization of $\phi_n(x)$ by means of the
partial Chebyshev polynomials.
The argument will be completed in the next section.

\subsection{Case II: Calculating $\tau_n$}

\begin{lemma}\label{03lem:eigenvectors of Pn}
Let $n\ge1$.
For $1\le k\le n$ let $\omega^{(n)}_k$ be the eigenvalue of $A_n$ 
defined as in \eqref{04eqn:ev(A_n)}.
Then the eigenspace associated to $\omega^{(n)}_k$
is orthogonal to ${\bf 1}$ if and only if $k$ is even.
\end{lemma}

\begin{proof}
It is easily verified that any eigenvector associated to
an eigenvalue $\alpha=\omega^{(n)}_k$ is
a constant multiple of $g=[g_1,\dots,g_n]^T$ given by
\[
g_l=\sin\frac{lk\pi}{n+1}\,,
\qquad 1\le l\le n\,.
\]
Moreover, with the help of an elementary formula
of trigonometric series we obtain
\[
\langle{\bf 1},g\rangle
=\sum_{l=1}^n \sin\frac{lk\pi}{n+1}
=\sin\dfrac{k\pi}{2}\sin\dfrac{nk\pi}{2(n+1)}
 \bigg(\sin\dfrac{k\pi}{2(n+1)}\bigg)^{-1}.
\]
Since $nk/(2(n+1))$ is not an integer for $1\le k\le n$, 
we see that
$\langle{\bf 1},g\rangle=0$ if and only if $k$ is even.
(In that case $g$ is skew-palindromic, see \cite{Doob-Haemers2002}.)
\end{proof}

Now we find the solutions $(f,g,\alpha,\mu)
\in \mathbb{R}\times\mathbb{R}^n\times\mathbb{R}\times\mathbb{R}$
to equations \eqref{04eqn:f=}, \eqref{04eqn:(A-a)g=},
\eqref{04eqn:eqn-Sa(3)} and \eqref{04eqn:eqn-Sa(4)}
with $\alpha\in(-\infty,-1)\cap \mathrm{ev}(A_n)$
by dividing them into two cases.

{\bfseries (Case II-1)} $\mu=0$.
Then $f=0$ by \eqref{04eqn:f=}, and \eqref{04eqn:(A-a)g=},
\eqref{04eqn:eqn-Sa(3)} and \eqref{04eqn:eqn-Sa(4)} become
\begin{equation}\label{03eqn:skew-palindromic}
A_ng=\alpha g,
\qquad \langle g,g\rangle=1,
\qquad \langle {\bf 1},g\rangle=0,
\end{equation}
respectively. 
It then follows from Lemma \ref{03lem:eigenvectors of Pn} that
$\alpha=\omega^{(n)}_k$ appears in the solutions
to \eqref{04eqn:f=}, \eqref{04eqn:(A-a)g=},
\eqref{04eqn:eqn-Sa(3)} and \eqref{04eqn:eqn-Sa(4)}
with $\alpha\in(-\infty,-1)\cap \mathrm{ev}(A_n)$
if and only if $\omega^{(n)}_k<-1$ and $k$ is even.
Note that for $1\le k\le n$ we have $\omega^{(n)}_k<-1$ 
if and only if $2(n+1)<3k$.

{\bfseries (Case II-2)} $\mu\neq0$.
Since the right-hand side of \eqref{04eqn:(A-a)g=} is not zero
in this case,
the equation has a solution if and only if
\begin{equation}\label{03eqn:ranks}
\mathrm{rank}(A_n-\alpha I)
=\mathrm{rank} [A_n-\alpha I, {\bf 1}],
\end{equation}
where $[A_n-\alpha I, {\bf 1}]$ is an $n\times (n+1)$ matrix 
obtained by adding ${\bf 1}$ as the $(n+1)$-th column vector.
We will show that \eqref{03eqn:ranks} fails for 
$\alpha=\omega_k^{(n)}$ with an odd $k$.

For $1\le i\le n$ let $v_i$ be the $i$th row vector of
$[A_n-\alpha I, {\bf 1}]$, that is,
\[
v_i=[(A_n-\alpha I)_{i1},\dots, (A_n-\alpha I)_{in},1].
\]
We first show that $v_1,v_2,\dots,v_n$ are linearly independent.
In fact, suppose that
\[
\sum_{i=1}^n \xi_i v_i=0,
\qquad \xi_1,\dots,\xi_n\in\mathbb{R}.
\]
Comparing the $j$th components of both sides,
we obtain
\begin{equation}\label{03eqn:in proof 8}
\sum_{i=1}^n \xi_i (A_n-\alpha I)_{ij}=0
\quad \text{for $1\le j\le n$},
\qquad
\sum_{i=1}^n \xi_i =0
\quad \text{for $j=n+1$}.
\end{equation}
We set $\xi=[\xi_1,\dots,\xi_n]^T$.
Since $A_n-\alpha I$ is a symmetric matrix, 
\eqref{03eqn:in proof 8} is equivalent to
\begin{equation}\label{03eqn:in proof 10}
(A_n-\alpha I)\xi=0,
\qquad
\langle {\bf 1},\xi\rangle=0.
\end{equation}
If $\alpha=\omega^{(n)}_k$ with an odd $k$,
then $\xi=0$ is the only vector that satisfies 
\eqref{03eqn:in proof 10} by Lemma \ref{03lem:eigenvectors of Pn}.
Therefore, $v_1,v_2,\dots,v_n$ are linearly independent.
Since they are the row vectors of the 
matrix $[A_n-\alpha I, {\bf 1}]$, we have 
$\mathrm{rank}[A_n-\alpha I, {\bf 1}]=n$.
On the other hand, since $\alpha$ is an eigenvalue of $A_n$, we have
$\mathrm{rank}(A_n I-\alpha)\le n-1$.
Thus, \eqref{03eqn:ranks} fails
and \eqref{04eqn:(A-a)g=} has no solution.
Consequently, $\alpha=\omega^{(n)}_k$ with an odd $k$
never appears in the solutions,
while $\alpha=\omega^{(n)}_k$ with
an even $k$ may appear therein.

Combining the results in (Case II-1) and (Case II-2),
we conclude that
\begin{equation}\label{04eqn:tau_n}
\tau_n
=\min\bigg\{\omega^{(n)}_k=2\cos \frac{k\pi}{n+1}\,;\, 
1\le k\le n,\,\, \text{$k$ is even}\bigg\}\cap (-\infty,-1),
\end{equation}
of which the explicit value is given by
\begin{align}
\tau_{2n+1}&=\omega^{(2n+1)}_{2n}=2\cos\frac{2n}{2n+2}\pi, 
\qquad
\qquad n\ge3,
\label{04eqn:tau_2n+1}\\
\tau_{2n}&=\omega^{(2n)}_{2n}=2\cos\frac{2n}{2n+1}\pi,
\qquad n\ge2.
\label{04eqn:tau_2n}
\end{align}
Note that $\tau_3$ and $\tau_5$ are not determined
because the right-hand side of \eqref{04eqn:tau_n} is empty.

For $n\ge2$ let $\beta_n$ be the minimal zero of $U^{\mathrm{e}}_n(x)$.
In view of Theorem \ref{02thm:product expressions} we have
\begin{equation}\label{03eqn:beta_n}
\beta_{2n}=\cos\frac{2n}{2n+1}\pi,
\qquad
\beta_{2n+1}=\cos\frac{2n}{2n+2}\pi,
\qquad n\ge1.
\end{equation}
Then, comparing \eqref{04eqn:tau_2n+1}, \eqref{04eqn:tau_2n}
and \eqref{03eqn:beta_n}, we come to the following assertion.

\begin{lemma}\label{04lem:tau_n}
For any $n\ge3$, $n\neq 3,5$ we have
$\tau_n=2\beta_n$.
\end{lemma}

\subsection{Main Results}

\begin{theorem}\label{04thm:main formula for fan 1}
For $n\ge1$ let $\alpha_n$ be the minimal zero of $\phi_n(x)$
defined in \eqref{03eqn:def of phi(x)}.
Then we have
\begin{equation}\label{04eqn:main formula for a fan graph}
\mathrm{QEC}(K_1+P_n)
=-2\alpha_n-2,
\qquad n\ge1.
\end{equation}
\end{theorem}

\begin{theorem}\label{04thm:main formula for fan 2}
For $n\ge1$ we have
\begin{gather}
\mathrm{QEC}(K_1+P_{2n})
=-4\sin^2\frac{\pi}{2(2n+1)}\,,
\label{04eqn:main formula for even n} \\
-4\sin^2\frac{\pi}{2(2n+2)}<\mathrm{QEC}(K_1+P_{2n+1})
< -4\sin^2\frac{\pi}{2(2n+3)}.
\label{04eqn:estimate for odd n}
\end{gather}
Moreover, we have
\begin{equation}\label{04eqn:QEC increases}
-\frac{1}{2}=\mathrm{QEC}(K_1+P_3)
<\mathrm{QEC}(K_1+P_4)
<\mathrm{QEC}(K_1+P_5)
<\dots \rightarrow 0.
\end{equation}
\end{theorem}

The proofs are based on Proposition \ref{03prop:preliminary formula},
and we need a careful analysis of zeroes of $\phi_n(x)$ and
$U^{\mathrm{e}}_n(x)$.
In the next section we will prove that
$\phi_n(x)$ is factored by $U^{\mathrm{e}}_n(x)$,
and in this line, we will obtain
explicit values and estimates of $\alpha_n$.

\begin{remark}\normalfont
A general formula is known for $\mathrm{QEC}(G_1+G_2)$
when both $G_1$ and $G_2$ are regular graphs 
\cite[Theorem 3.1]{Lou-Obata-Huang2022}.
The results in \eqref{04eqn:main formula for even n}
and \eqref{04eqn:estimate for odd n} show that
the general formula remains valid for $K_1+P_{2n}$ 
but not for $K_1+P_{2n+1}$.
\end{remark}

\begin{remark}\normalfont
The list of $\mathrm{QEC}(G)$ is available for
small graphs $G$ in \cite{Obata-Zakiyyah2018},
where $K_1+P_n$ for $n=1,2,3,4$ appears as
G2.1, G3.2, G4.5 and G5.16, respectively.
\end{remark}

\section{Key Polynomials $\phi_n(x)$}
\label{sec:New Polynomials}

In this section we examine the zeroes of the key polynomial:
\[
\phi_n(x)
=((n+1)x^2-3x-n)U_n(x)+(x+1)(U_{n-1}(x)+1),
\qquad n\ge0,
\]
introduced in \eqref{03eqn:def of phi(x)}
and complete the proof of the main results
(Theorems \ref{04thm:main formula for fan 1} and
\ref{04thm:main formula for fan 2}).
For that purpose, with each $n\ge0$ we associate
polynomials $S_n(x)$ defined by
\begin{align}
S_{2n}(x)&=(2nx+x+2n-1)U_n(x)-(2nx+3x+2n+1)U_{n-1}(x),
\label{04eqn:def of S_2n} \\
S_{2n+1}(x)&=2(2nx^2+2x^2+2nx-x-1)U_n(x)-2(2nx+3x+2n+1)U_{n-1}(x).
\label{04eqn:def of S_2n+1}
\end{align}

\subsection{Factorization}

\begin{theorem}\label{04thm:factorization pf phi_n(x)}
For $n\ge0$ we have a factorization:
\begin{equation}\label{04eqn:factorization pf phi_n(x)}
\phi_n(x)=(x-1)\cdot U_{n}^{\mathrm{e}}(x)\cdot S_n(x).
\end{equation}
\end{theorem}

\begin{proof}
Let $n\ge0$.
In view of \eqref{02eqn:factorization of U_2n and U_2n+1},
\eqref{for:unplus1} and \eqref{for:ueuoeven} we obtain
\begin{align*}
U_{2n}
&=(U_n+U_{n-1})(U_n-U_{n-1})=U^{\mathrm{e}}_{2n}(U_n-U_{n-1}),\\
U_{2n-1}+1
&=(U_n+U_{n-1})(U_{n-1}-U_{n-2})
=U^{\mathrm{e}}_{2n}(U_{n-1}-U_{n-2}).
\end{align*}
Then we have
\begin{align*}
\phi_{2n}(x)
&=((2n+1)x^2-3x-2n)U_{2n}+(x+1)(U_{2n-1}+1) \\
&=U^{\mathrm{e}}_{2n}
  \big\{((2n+1)x^2-3x-2n)(U_n-U_{n-1})+(x+1)(U_{n-1}-U_{n-2})\big\}.
\end{align*}
Now we apply $U_{n-2}=2x U_{n-1}-U_n$ to get
\begin{align*}
\phi_{2n}(x)
&=U^{\mathrm{e}}_{2n}
 \big\{((2n+1)x^2-2x-2n+1)U_n-((2n+3)x^2-2x-2n-1)U_{n-1}\big\} \\
&=U^{\mathrm{e}}_{2n}(x-1)
 \big\{(2nx+x+2n-1)U_n-(2nx+3x+2n+1)U_{n-1}\big\} \\
&=(x-1)U^{\mathrm{e}}_{2n} S_{2n},
\end{align*}
which proves \eqref{04eqn:factorization pf phi_n(x)} for an even $n$.
By a similar calculation using
\begin{align*}
U_{2n+1}
&=U_n(U_{n+1}-U_{n-1})=U^{\mathrm{e}}_{2n+1}(U_{n+1}-U_{n-1}),\\
U_{2n}+1
&=U_n(U_n-U_{n-2})=U^{\mathrm{e}}_{2n+1}(U_n-U_{n-2}),
\end{align*}
which follow from 
\eqref{02eqn:factorization of U_2n and U_2n+1},
\eqref{for:unplus1} and \eqref{for:ueuoodd},
we come to \eqref{04eqn:factorization pf phi_n(x)} for an odd $n$.
\end{proof}

\subsection{Zeroes of $S_n(x)$}

\begin{lemma}\label{04lem:special value of S2n(x)}
Let $n\ge1$.
For $1\le k\le n$, we have
\begin{align}
S_{2n}\left(\cos\frac{2k\pi}{2n+1}\right)
&=2(-1)^k \left(\cos\dfrac{k\pi}{2n+1}\right)^{-1}
 \left(n+(n+1)\cos\dfrac{2k\pi}{2n+1}\right),
\label{04eqn:special values of Un-1 even 2k} \\
S_{2n}\left(\cos\frac{2k-1}{2n+1}\pi\right)
&=(-1)^k\left(\sin\dfrac{2k-1}{2(2n+1)}\pi\right)^{-1}
 \left(1+\cos\dfrac{2k-1}{2n+1}\pi\right),
\label{04eqn:special values of Un-1 odd 2k-1}
\end{align}
and 
\begin{equation}\label{04eqn:S2n(1) and S2n(-1)}
S_{2n}(1)=0,
\qquad
S_{2n}(-1)=2(-1)^{n-1}(2n+1).
\end{equation}
\end{lemma}

\begin{proof}
By definition \eqref{04eqn:def of S_2n} for any $l\in\mathbb{Z}$ we have
\begin{align}
S_{2n}\left(\cos\frac{l\pi}{2n+1}\right)
&=\left((2n+1)\cos\frac{l\pi}{2n+1}+2n-1\right)
  U_n\left(\cos\frac{l\pi}{2n+1}\right) 
\nonumber \\
&\qquad  -\left((2n+3)\cos\frac{l\pi}{2n+1}+2n+1\right)
  U_{n-1}\left(\cos\frac{l\pi}{2n+1}\right).
\label{03eqn:in proof Lemma 3.2}
\end{align}
Setting $l=0$ and $l=2n+1$, we obtain
\eqref{04eqn:S2n(1) and S2n(-1)} by using $U_n(\pm1)=(\pm1)^n (n+1)$. 
For $1\le k\le n$, \eqref{04eqn:special values of Un-1 even 2k} 
and \eqref{04eqn:special values of Un-1 odd 2k-1} follow by
setting $l=2k$ and $l=2k-1$ in \eqref{03eqn:in proof Lemma 3.2}
and applying the following formulas:
\begin{gather*}
U_n\bigg(\cos\frac{2k\pi}{2n+1}\bigg)
=\frac{(-1)^k}{2}\bigg(\cos\dfrac{k\pi}{2n+1}\bigg)^{-1},
\\
U_{n-1}\bigg(\cos\frac{2k\pi}{2n+1}\bigg)
=\frac{(-1)^{k-1}}{2}\bigg(\cos\dfrac{k\pi}{2n+1}\bigg)^{-1},
\\
U_n\bigg(\cos\frac{2k-1}{2n+1}\pi\bigg)
=U_{n-1}\bigg(\cos\frac{2k-1}{2n+1}\pi\bigg)
=\frac{(-1)^{k-1}}{2}\bigg(\sin\dfrac{2k-1}{2(2n+1)}\pi\bigg)^{-1},
\end{gather*}
which are verified directly by the original definition 
\eqref{01eqn:def of U_n(x)}.
\end{proof}

\begin{theorem}\label{04thm:roots of S2n(x)}
Let $n\ge1$ and set
\[
\xi_0=1, 
\qquad
\xi_{n+1}=-1,
\qquad
\xi_k=\cos\frac{2k-1}{2n+1}\pi,
\qquad 1\le k\le n.
\]
Then
$1=\xi_0>\xi_1>\cdots>\xi_k>\xi_{k+1}>\cdots>\xi_n>\xi_{n+1}=-1$
and $S_{2n}(x)$ has a unique zero in the interval
$(\xi_{k+1},\xi_k)$ for any $1\le k \le n$.
These $n$ zeroes together with $x=\xi_0=1$ exhaust
the zeroes of $S_{2n}(x)$, which are all simple.
(The statement remains true for $n=0$ 
by natural interpretation since $x=1$ is the
only zero of $S_0(x)=x-1$.)
\end{theorem}

\begin{proof}
We will show that
\begin{equation}\label{06eqn:interlacing S2n}
S_{2n}(\xi_k)S_{2n}(\xi_{k+1})<0,
\qquad 1\le k \le n.
\end{equation}
It follows from Lemma \ref{04lem:special value of S2n(x)} that
\[
S_{2n}(\xi_k)
=(-1)^k\bigg(\sin\dfrac{2k-1}{2(2n+1)}\pi\bigg)^{-1}
 \bigg(1+\cos\dfrac{2k-1}{2n+1}\pi\bigg),
\qquad 1\le k \le n,
\]
and hence $S_{2n}(\xi_k)S_{2n}(\xi_{k+1})<0$ for any $1\le k \le n-1$.
Namely, \eqref{06eqn:interlacing S2n} is valid for $1\le k \le n-1$.
Moreover, since 
\[
S_{2n}(\xi_{n+1})=S_{2n}(-1)=2(-1)^{n-1}(2n+1)
\]
by Lemma \ref{04lem:special value of S2n(x)},
we have
\begin{align*}
&S_{2n}(\xi_n)S_{2n}(\xi_{n+1})
\\
&=(-1)^n\bigg(\sin\dfrac{2n-1}{2(2n+1)}\pi\bigg)^{-1}
 \bigg(1+\cos\dfrac{2n-1}{2n+1}\pi\bigg)
 \times (-1)^{n-1}(2n+1)
<0,
\end{align*}
which means that 
\eqref{06eqn:interlacing S2n} is valid also for $k=n$.
We see from 
\eqref{06eqn:interlacing S2n} with the intermediate value theorem
that $S_{2n}(x)$ has at least one zero in the interval
$(\xi_{k+1},\xi_k)$ for each $1\le k\le n$.
On the other hand, we know that $S_{2n}(1)=0$.
Since $S_{2n}$ is a polynomial of order $n+1$, 
each interval $(\xi_{k+1},\xi_k)$ contains exactly one zero
for each $1\le k\le n$,
and all zeroes of $S_{2n}(x)$ are simple.
\end{proof}

\begin{lemma}\label{04lem:special value of S2n+1(x)}
Let $n\ge1$.
For $1\le k\le n$ we have
\begin{equation}\label{06eqn:special values of S2n+1 even 2k}
S_{2n+1}\bigg(\cos\frac{2k\pi}{2n+2}\bigg)
=2(-1)^k \bigg(2n+1+(2n+3)\cos\dfrac{k\pi}{n+1}\bigg),
\end{equation}
and for $1\le k\le n+1$ we have
\begin{equation}\label{06eqn:special values of S2n+1 odd 2k-1}
S_{2n+1}\bigg(\cos\frac{2k-1}{2n+2}\pi\bigg)
=2(-1)^k\bigg(\sin\dfrac{2k-1}{2n+2}\pi\bigg)^{-1}
 \bigg(1+\cos\dfrac{2k-1}{2n+2}\pi\bigg)^2.
\end{equation}
Moreover,
\begin{equation}\label{06eqn:S2n+1(1) and S2n+1(-1)}
S_{2n+1}(1)=0,
\qquad
S_{2n+1}(-1)=4(-1)^n.
\end{equation}
\end{lemma}

The proof is computational 
as in that of Lemma \ref{04lem:special value of S2n(x)}
and is omitted.

\begin{theorem}\label{06thm:roots of S2n+1(x)}
Let $n\ge0$ and set
\[
\xi_0=1, 
\qquad
\xi_{n+2}=-1,
\qquad
\xi_k=\cos\frac{2k-1}{2n+2}\pi,
\qquad 1\le k\le n+1.
\]
Then 
$1=\xi_0>\xi_1>\cdots>\xi_k>\xi_{k+1}>\cdots>\xi_{n+1}>\xi_{n+2}=-1$
and $S_{2n+1}(x)$ has a unique zero in the interval
$(\xi_{k+1},\xi_k)$ for any $1\le k \le n+1$.
These $n+1$ zeroes together with $x=\xi_0=1$
exhaust the zeroes of $S_{2n+1}(x)$, which are all simple.
\end{theorem}

\begin{proof}
For $n=0$ we have $\xi_0=1$, $\xi_1=0$ and $\xi_2=-1$.
In view of $S_1(x)=2(2x+1)(x-1)$ we see that
$S_1(x)$ has a unique zero (in fact $x=-1/2$) 
in the interval $(\xi_2,\xi_1)=(-1,0)$.
Then the assertion for $n=0$ is obvious.
Now we assume that $n\ge1$.
We will show that
\begin{equation}\label{06eqn:interlacing S2n+1}
S_{2n+1}(\xi_k)S_{2n+1}(\xi_{k+1})<0,
\qquad 1\le k \le n+1.
\end{equation}
In fact, \eqref{06eqn:interlacing S2n+1} is verified for $1\le k\le n$
by using
\[
S_{2n+1}(\xi_k)
=2(-1)^k\bigg(\sin\dfrac{2k-1}{2n+2}\pi\bigg)^{-1}
 \bigg(1+\cos\dfrac{2k-1}{2n+2}\pi\bigg)^2,
\qquad 1\le k\le n+1,
\]
see \eqref{06eqn:special values of S2n+1 odd 2k-1}
in Lemma \ref{04lem:special value of S2n+1(x)}.
Moreover, since $S_{2n+1}(-1)=4(-1)^n$ 
by in Lemma \ref{04lem:special value of S2n+1(x)},
we have
\begin{align*}
&S_{2n+1}(\xi_{n+1})S_{2n+1}(\xi_{n+2})
\\
&\qquad =2(-1)^{n+1}\bigg(\sin\dfrac{2n+1}{2n+2}\pi\bigg)^{-1}
 \bigg(1+\cos\dfrac{2n+1}{2n+2}\pi\bigg)^2
 \times 4(-1)^n
<0,
\end{align*}
which means that 
\eqref{06eqn:interlacing S2n+1} is valid also for $k=n+1$.
Then the intermediate value theorem says that
$S_{2n+1}(x)$ has at least one zero in the interval
$(\xi_{k+1},\xi_k)$ for each $1\le k\le n+1$.
Since $S_{2n+1}(1)=0$ and $S_{2n+1}(x)$ is a polynomial of order $n+2$,
each interval $(\xi_{k+1},\xi_k)$ contains exactly one zero
for each $1\le k\le n+1$,
and every zero of $S_{2n+1}(x)$ is simple.
\end{proof}

\begin{remark}\label{03rem:1 is zero of S_n(x)}
\normalfont
Since $S_n(1)=0$ for any $n\ge0$ by
Lemmas \ref{04lem:special value of S2n(x)}
and \ref{04lem:special value of S2n+1(x)},
we see that $S_n(x)$ is factored by $x-1$ and so is $\phi_n(x)$ by $(x-1)^2$,
see Theorem \ref{04thm:factorization pf phi_n(x)}.
After routine calculus we may write down $S_n(x)/(x-1)$ in terms of Chebyshev series,
however, no concise form has been obtained yet.
\end{remark}

\subsection{Comparison of Minimal Zeroes of $S_n(x)$ and $U^{\mathrm{e}}_n(x)$}

For $n\ge2$ let $\beta_n$ and $\gamma_n$ be 
the minimal zeroes of
$U^{\mathrm{e}}_n(x)$ and $S_n(x)$, respectively.
Note that $\beta_1$ makes no sense 
because $U^{\mathrm{e}}_1(x)=1$ has no zero,
while $\gamma_1=-1/2$ by $S_1(x)=2(2x+1)(x-1)$.

\begin{lemma}\label{03lem:elementary calculus}
We have
\begin{equation}\label{03eqn:elementary inequality 0}
\frac{1-x}{1+x}\le \cos \pi x,
\qquad
0\le x \le \frac13,
\end{equation}
where equality holds only for $x=0$ and $x=1/3$.
\end{lemma}

\begin{proof}
By elementary calculus.
For example, we observe
that the left-hand side of \eqref{03eqn:elementary inequality 0}
is convex on the interval $(0,1/3)$ while the right-hand side
is concave thereon.
\end{proof}

\begin{theorem}\label{03thm:comparison beta and gamma}
We have
\begin{gather}
\beta_2=\gamma_2=-\frac12=\cos\frac{2}{3}\pi,
\label{03eqn:alpha and beta for n=2} 
\\
-1<\gamma_{2n+1}<\cos\frac{2n+1}{2n+2}\pi
<\beta_{2n+1}=\cos\frac{2n}{2n+2}\pi,
\qquad n\ge1,
\label{03eqn:alpha and beta for n odd}
\\
-1<\beta_{2n}=\cos\frac{2n}{2n+1}\pi
<\gamma_{2n}<\cos\frac{2n-1}{2n+1}\pi,
\qquad n\ge2.
\label{03eqn:alpha and beta for n even}
\end{gather}
\end{theorem}

\begin{proof}
For $n\ge2$ the explicit value of $\beta_n$ is already mentioned 
in \eqref{03eqn:beta_n}.
Then for \eqref{03eqn:alpha and beta for n=2} 
we need only to note that $\gamma_2=-1/2$,
which follows from $S_2(x)=3(2x+1)(x-1)$.
Moreover, it follows from Theorem \ref{06thm:roots of S2n+1(x)} that
\begin{equation}\label{06eqn:minimal zeros of S2n+1(x)}
-1<\gamma_{2n+1}<\cos\frac{2n+1}{2n+2}\pi,
\qquad n\ge0,
\end{equation}
from which we obtain \eqref{03eqn:alpha and beta for n odd}.

For \eqref{03eqn:alpha and beta for n even}
we need to compare $\beta_{2n}$ and $\gamma_{2n}$ for $n\ge2$.
In view of Lemma \ref{04lem:special value of S2n(x)}
we have
\begin{align}
S_{2n}\left(\cos\frac{2n\pi}{2n+1}\right)
&=2(-1)^n \left(\cos\dfrac{n\pi}{2n+1}\right)^{-1}
 \left(n+(n+1)\cos\dfrac{2n\pi}{2n+1}\right),
\label{03eqn:in proof Thm 3.9 (1)}
\\
S_{2n}\left(\cos\frac{2n-1}{2n+1}\pi\right)
&=(-1)^n\left(\sin\dfrac{2n-1}{2(2n+1)}\pi\right)^{-1}
 \left(1+\cos\dfrac{2n-1}{2n+1}\pi\right).
\label{03eqn:in proof Thm 3.9 (2)}
\end{align}
Hence the signature of the product
of \eqref{03eqn:in proof Thm 3.9 (1)}
and \eqref{03eqn:in proof Thm 3.9 (2)}
coincides with that of
\begin{equation}\label{04eqn:in proof Theorem 3.11 (2)}
n+(n+1)\cos\dfrac{2n\pi}{2n+1}
=n-(n+1)\cos\dfrac{\pi}{2n+1}\,.
\end{equation}
Setting $x=1/(2n+1)$ in Lemma \ref{03lem:elementary calculus},
we see that \eqref{04eqn:in proof Theorem 3.11 (2)} is
negative for any $n\ge2$,
and so is the product of \eqref{03eqn:in proof Thm 3.9 (1)}
and \eqref{03eqn:in proof Thm 3.9 (2)}.
Then by the intermediate value theorem 
there is a zero of $S_{2n}(x)$ in the interval 
$(\cos 2n\pi/(2n+1), \cos((2n-1)\pi)/(2n+1))$.
Indeed, this zero is $\gamma_{2n}$
as Theorem \ref{04thm:roots of S2n(x)} says that
there is a unique zero of $S_{2n}(x)$ in the interval 
$(-1, \cos((2n-1)\pi)/(2n+1))$.
Thus, combining \eqref{03eqn:beta_n} we 
obtain \eqref{03eqn:alpha and beta for n even}.
\end{proof}

\subsection{Minimal Zero of $\phi_n(x)$}

Recall that the minimal zeroes of
$\phi_n(x)=(x-1)\cdot U^{\mathrm{e}}_n(x)\cdot S_n(x)$,
$U^{\mathrm{e}}_n(x)$ and $S_n(x)$
are denoted by
$\alpha_n$, $\beta_n$ and $\gamma_n$, respectively.
By Remark \ref{03rem:1 is zero of S_n(x)} 
any zero of $\phi_n(x)$ is a zero of 
$U^{\mathrm{e}}_n(x)$ or of $S_n(x)$.
We obtain $\alpha_0=1$ from $\phi_0(x)=(x-1)^2$ and
$\alpha_1=-1/2$ from $\phi_1(x)=2(2x+1)(x-1)^2$.
Then the following result is immediate from
Theorem \ref{03thm:comparison beta and gamma}.

\begin{theorem}\label{04thm:alpha_n minimal zero of phi_n}
We have $\alpha_0=1$, $\alpha_1=-1/2$ and
\begin{equation}\label{03eqn:alpha_n, beta_n, gamma_n}
\alpha_n=\min\{\beta_n,\gamma_n\}
=\begin{cases}
\beta_n &\text{for an even $n\ge2$}, \\
\gamma_n &\text{for an odd $n\ge3$}.
\end{cases}
\end{equation}
\end{theorem}

\begin{lemma}\label{03lem:S_2n+1(beta_2n+2)}
For $n\ge1$ we have
\begin{equation}\label{03eqn:S_2n+1(cos 2n+2/2n+3)}
S_{2n+1}\bigg(\cos\dfrac{2n+2}{2n+3}\pi\bigg)
=4(-1)^n(n+2)\sin\frac{\pi}{2(2n+3)}
\bigg(\cos\frac{\pi}{2n+3}-\frac{n+1}{n+2}\bigg).
\end{equation}
\end{lemma}

\begin{proof}
In an analogous manner as in Lemma \ref{04lem:special value of S2n+1(x)},
we obtain \eqref{03eqn:S_2n+1(cos 2n+2/2n+3)} by
direct calculation using 
\[
U_{n-1}\bigg(\cos\dfrac{2n+2}{2n+3}\pi\bigg)
=(-1)^{n+1} \frac{\cos 3\omega}{\sin 2\omega}\,,
\qquad
U_n\bigg(\cos\dfrac{2n+2}{2n+3}\pi\bigg)
=(-1)^n \frac{\cos \omega}{\sin 2\omega}\,,
\]
where $\omega=\pi/(2(2n+3))$.
\end{proof}

\begin{lemma}\label{04lem:comparison of beta and gamma}
For $n\ge1$ we have
\begin{equation}\label{03eqn:beta<gamma<beta}
\beta_{2n+2}<\gamma_{2n+1}<\beta_{2n}\,,
\qquad n\ge1.
\end{equation}
In particular,
\begin{equation}\label{03eqn:alpha_n strict decrease}
1=\alpha_0>\alpha_1=\alpha_2=-\frac12
>\alpha_3>\alpha_4>\alpha_5>\cdots\rightarrow -1.
\end{equation}
\end{lemma}

\begin{proof}
The right-half of \eqref{03eqn:beta<gamma<beta} follows immediately
from Theorem \ref{03thm:comparison beta and gamma} by
observing that
\begin{equation}\label{06eqn:in proof Thm 6.18(1)}
\gamma_{2n+1}<\cos\frac{2n+1}{2n+2}\pi
<\cos\frac{2n}{2n+1}\pi=\beta_{2n}\,.
\end{equation}
The left-half of \eqref{03eqn:beta<gamma<beta} requires
more careful consideration.
Let $\delta_{2n+1}$ be the second smallest zero of
$S_{2n+1}(x)$, while the minimal one has been denoted by $\gamma_{2n+1}$.
It follows from Theorem \ref{06thm:roots of S2n+1(x)} that
\begin{equation}\label{03eqn:in proof Thm 3.10 (1)}
-1<\gamma_{2n+1}<\cos\frac{2n+1}{2n+2}\pi
<\delta_{2n+1}<\cos\frac{2n-1}{2n+2}\pi.
\end{equation}
Since $S_{2n+1}(x)$ is a polynomial of order $n+2$,
the coefficient of the leading term is positive
and every zero of it is simple,
we have
\begin{align}
&S_{2n+1}(x)<0 \,\,\text{for $x<\gamma_{2n+1}$},
\quad
S_{2n+1}(x)>0 \,\, \text{for $\gamma_{2n+1}<x<\delta_{2n+1}$},
\quad \text{if $n\ge1$ is odd},
\label{03eqn:S_{2n+1} for odd n} \\
&S_{2n+1}(x)>0 \,\, \text{for $x<\gamma_{2n+1}$},
\quad
S_{2n+1}(x)<0 \,\, \text{for $\gamma_{2n+1}<x<\delta_{2n+1}$},
\quad \text{if $n\ge2$ is even}.
\label{03eqn:S_{2n+1} for even n}
\end{align}
On the other hand, using the exact value of $\beta_{2n+2}$ in
Theorem \ref{03thm:comparison beta and gamma} and
combining \eqref{03eqn:in proof Thm 3.10 (1)}, we obtain
\begin{equation}\label{07eqn:beta_2n+2<r}
\beta_{2n+2}=\cos\frac{2n+2}{2n+3}\pi
<\frac{2n+1}{2n+2}\pi<\delta_{2n+1},
\end{equation}
Therefore, for $\beta_{2n+2}<\gamma_{2n+1}$ it is sufficient to
show that $S_{2n+1}(\beta_{2n+2})<0$ or
$S_{2n+1}(\beta_{2n+2})>0$ according as
$n\ge1$ is odd or $n\ge2$ is even,
see \eqref{03eqn:S_{2n+1} for odd n}
and \eqref{03eqn:S_{2n+1} for even n}.
Thus, it is sufficient to show that
\begin{equation}\label{03eqn:in proof Thm 3.10 (7)}
(-1)^n S_{2n+1}\bigg(\cos\dfrac{2n+2}{2n+3}\pi\bigg)>0,
\qquad n\ge1.
\end{equation}
Indeed, \eqref{03eqn:in proof Thm 3.10 (7)} follows from
Lemmas \ref{03lem:S_2n+1(beta_2n+2)} and
\ref{03lem:elementary calculus}.
Finally, \eqref{03eqn:alpha_n strict decrease} follows easily
by Theorem \ref{04thm:alpha_n minimal zero of phi_n}.
\end{proof}

\subsection{Proofs of Main Results}

We start with the evaluation of $\sigma_n$ in
Lemma \ref{03lem:preparatory formula for sigma_n}.

\begin{lemma}\label{04lem:sigma_n}
For $n\ge3$ we have $\sigma_n=2\gamma_n$.
Moreover, we have 
\begin{gather}
\sigma_{2n+1}<\omega^{(2n+1)}_{2n+1}<\omega^{(2n+1)}_{2n},
\qquad n\ge1, 
\label{04eqn:sigma_2n+1 estimate} \\
\omega^{(2n)}_{2n}<\sigma_{2n}<\omega^{(2n)}_{2n-1},
\qquad n\ge2,
\label{04eqn:sigma_2n estimate}
\end{gather}
where $\omega^{(n)}_k$ is the $k$th largest eigenvalue
of $A_n$ as described in \eqref{04eqn:ev(A_n)}.
\end{lemma}

\begin{proof}
By virtue of Theorem \ref{04thm:factorization pf phi_n(x)} we
first note that
\[
\phi_n\left(\frac{x}{2}\right)
=\left(\frac{x}{2}-1\right)
 \cdot U^{\mathrm{e}}_n\left(\frac{x}{2}\right)
 \cdot S_n\left(\frac{x}{2}\right).
\]
Since $\det(x-A_n)=U_n(x/2)
=U^{\mathrm{e}}_n(x/2)\cdot U^{\mathrm{o}}_n(x/2)$,
any root of $U^{\mathrm{e}}_n(x/2)=0$ is an eigenvalue of $A_n$
and does not contribute to $\sigma_n$. 
On the other hand, by Theorem \ref{03thm:comparison beta and gamma}
the minimal root of $S_n(x/2)=0$ is given by $2\gamma_n$,
which satisfies
\begin{equation}\label{04eqn:2gamma for n odd}
2\gamma_{2n+1}
<2\cos\frac{2n+1}{2n+2}\pi=\omega^{(2n+1)}_{2n+1}
<2\cos\frac{2n}{2n+2}\pi=\omega^{(2n+1)}_{2n},
\qquad n\ge1,
\end{equation}
\begin{equation}\label{04eqn:2gamma for n even}
\omega^{(2n)}_{2n}=2\cos\frac{2n}{2n+1}\pi
<2\gamma_{2n}<2\cos\frac{2n-1}{2n+1}\pi=\omega^{(2n)}_{2n-1},
\qquad n\ge2.
\end{equation}
In particular, $2\gamma_n$ is not eigenvalue of $A_n$ for any $n\ge3$.
Consequently, 
\[
\sigma_n
=\min\{\alpha\in(-\infty,-1)\backslash \mathrm{ev}(A_n)\,;\,
\phi_n(\alpha/2)=0\}
=2\gamma_n.
\]
Inequalities \eqref{04eqn:sigma_2n+1 estimate}
and \eqref{04eqn:sigma_2n estimate} are immediate from
\eqref{04eqn:2gamma for n odd} and
\eqref{04eqn:2gamma for n even}, respectively.
\end{proof}

\begin{proof}[Proof of Theorem \ref{04thm:main formula for fan 1}]
We see from Lemmas \ref{04lem:tau_n} and \ref{04lem:sigma_n} that
\[
\min\{\sigma_n,\tau_n\}
=\min\{2\gamma_n,2\beta_n\}
=2\min\{\gamma_n,\beta_n\}
=2\alpha_n,
\qquad n\ge3, \quad n\neq 3,5,
\]
where Theorem \ref{04thm:alpha_n minimal zero of phi_n} is also
taken into account.
For $n=3,5$ since $\tau_n$ is not defined,
we have $\sigma_n=2\gamma_n=2\alpha_n$.
Therefore, for any $n\ge3$ we have
$\min\{\sigma_n,\tau_n\}=2\alpha_n$.
Then, using Proposition \ref{03prop:preliminary formula},
we conclude that $\tilde{\alpha}_n=\min\{\sigma_n,\tau_n\}=2\alpha_n$
for $n\ge3$.
For $n=1,2$ the relation $\tilde{\alpha}_n=2\alpha_n$ remains valid
by directly observing that $\alpha_1=\alpha_2=-1/2$ and
$\Tilde{\alpha}_1=\Tilde{\alpha}_2=-1$.
Consequently, the preliminary formula
\eqref{03eqn:main formula for QEC(K1+Pn)}
in Proposition \ref{03prop:preliminary formula} 
leads to $\mathrm{QEC}(K_1+P_n)=-2\alpha_n-2$ for $n\ge1$.
\end{proof}

\begin{proof}[Proof of Theorem \ref{04thm:main formula for fan 2}]
Relation \eqref{04eqn:main formula for even n} follows
from $\mathrm{QEC}(K_1+P_{2n})=-2\alpha_{2n}-2$ in
Theorem \ref{04thm:main formula for fan 1},
where the value of $\alpha_{2n}=\beta_{2n}=\cos (2n\pi)/(2n+1)$
is known by Theorem \ref{04thm:alpha_n minimal zero of phi_n}
and \eqref{03eqn:beta_n}.
Since 
\[
\cos\frac{2n+2}{2n+3}\pi
=\beta_{2n+2}<\gamma_{2n+1}=\alpha_{2n+1}<\cos\frac{2n+1}{2n+2}\pi
\]
by Lemma \ref{04lem:comparison of beta and gamma}
and Theorem \ref{03thm:comparison beta and gamma},
we obtain
\[
-2\cos\frac{2n+1}{2n+2}\pi-2
<\mathrm{QEC}(P_{2n+1})
<-2\cos\frac{2n+2}{2n+3}\pi-2,
\]
from which \eqref{04eqn:estimate for odd n} follows.
\end{proof}

{\bfseries Acknowledgements} \enspace
NO is supported by the JSPS Grant-in-Aid for Scientific Research
23K03126.



\begin{thebibliography}{99}

\bibitem{Alfakih2018}
Alfakih  AY. 2018
\textit{Euclidean Distance Matrices and their Applications in Rigidity Theory}.
Springer, Cham.

\bibitem{Balaji-Bapat2007}
Balaji R, Bapat RB. 2007
On Euclidean distance matrices.
\textit{Linear Algebra Appl.} {\bfseries 424}, 108--117.

\bibitem{Balasubramanian2023}
Balasubramanian K. 2023
Orthogonal polynomials through complex matrix graph theory.
\textit{J. Math. Chem.} {\bfseries 61}, 144--165.

\bibitem{Balasubramanian2025}
Balasubramanian K. 2025
Matching polynomials of symmetric, semisymmetric, double group graphs, 
polyacenes, wheels, fans, and symmetric solids in third and higher dimensions.
\textit{Symmetry} {\bfseries 2025}, 17, 133.

\bibitem{Bapat2010}
Bapat RB. 2010
\textit{Graphs and Matrices}.
Springer, Hindustan Book Agency, New Delhi.

\bibitem{Baskoro-Obata2021}
Baskoro ET, Obata N. 2021
Determining finite connected graphs 
along the quadratic embedding constants of paths.
\textit{Electron. J. Graph Theory Appl.} {\bfseries 9}. 539--560.

\bibitem{Baskoro-Obata2024}
Baskoro ET, Obata N. 2024 
A classification of graphs through quadratic embedding constants
and clique graph insights.
\textit{Commun. Combinatorics Optimization}, in press.

\bibitem{Biggs1993}
Biggs N. 1993
\textit{Algebraic Graph Theory}.
Cambridge University Press, Cambridge.

\bibitem{Blumenthal1953}
Blumenthal LM. 1953
\textit{Theory and Applications of Distance Geometry}.
Clarendon Press, Oxford.

\bibitem{Brouwer-Haemers2012}
Brouwer AE, Haemers WH. 2012
\textit{Spectra of Graphs}.
Springer, New York.

\bibitem{Choudhury-Nandi2023}
Choudhury PN, Nandi R. 2024
Quadratic embedding constants of graphs: Bounds
and distance spectra.
\textit{Linear Algebra Appl.} {\bfseries 680}, 108--125.

\bibitem{Cvetkovic-Rowlinson-Simic2010}
Cvetkovi\'{c} D, Rowlinson P, Simi\'{c} S. 2010
\textit{An Introduction to the Theory of Graph Spectra}.
Cambridge University Press, Cambridge.

\bibitem{Deza-Laurent1997}
Deza MM, Laurent M. 1997
\textit{Geometry of Cuts and Metrics}.
Springer-Verlag, Berlin.

\bibitem{Doob-Haemers2002}
Doob M, Haemers WH. 2002
The complement of the path is determined by its spectrum.
\textit{Linear Algebra Appl.} {\bfseries 356}, 57--65.

\bibitem{Falcon-Venkatachalam-Gowri-Nandini2021}
Falc\'on RM, Venkatachalam M, Gowri S, Nandini G. 2021
On the r-dynamic coloring of some fan graph families.
\textit{An. \c{S}t. Univ. Ovidius Constan\c{t}a} {\bfseries 29} (2021), 151--181.

\bibitem{Fowler-Caporossi-Hansen2001}
Fowler PW, Caporossi G, Hansen P. 2001
Distance matrices, Wiener indices, and related invariants of fullerenes.
\textit{J. Phys. Chem. A} {\bfseries 105}, 6232--6242.

\bibitem{Hao-Li-Zhang2022}
Hao C, Li S, Zhang L. 2022
An inverse formula for the distance matrix of a fan graph.
\textit{Linear Multilinear Algebra} {\bfseries 70}, 7807--7824.

\bibitem{Hosoya-Randic1983}
Hosoya H, Randi\'{c} M. 1983
Analysis of the topological dependency of
the characteristic polynomial in its Chebyshev expansion.
\textit{Theoret. Chim. Acta (Berl.)}
{\bfseries 63} 473--495.

\bibitem{Irawan-Sugeng2021}
Irawan W, Sugeng KA. 2021
Quadratic embedding constants of hairy cycle graphs,
\textit{Journal of Physics: Conference Series}
{\bfseries 1722}, 012046.

\bibitem{Jaklic-Modic2013}
Jakli\v{c} G, Modic J. 2013
On Euclidean distance matrices of graphs.
\textit{Electron. J. Linear Algebra} {\bfseries 26}, 574--589.

\bibitem{Keri2022}
K\'eri G. 2022
The factorization of compressed Chebyshev polynomials and
other polynomials related to multiple-angle formulas.
\textit{Ann. Univ. Sci. Budapest, Sect. Comp.} {\bfseries 53}, 93--108.

\bibitem{Liberti-Lavor-Maculan-Mucherino2014}
Liberti L, Lavor G, Maculan N, Mucherino A. 2014
Euclidean distance geometry and applications.
\textit{SIAM Rev.} {\bfseries 56} (2014), 3--69.

\bibitem{Liu-Yuan-Das2020}
Liu M, Yuan Y, Das ChK. 2020
The fan graph is determined by its signless Laplacian spectrum.
\textit{Czechoslovak Math. J.} {\bfseries 70}, 21--31.

\bibitem{Lou-Obata-Huang2022}
Lou ZZ, Obata N, Huang QX. 2022
Quadratic embedding constants of graph joins.
\textit{Graphs Combin.} {\bfseries 38}, 161 (22 pages).

\bibitem{Mason-Handscomb2003}
Mason JC, Handscomb DC. 2003
\textit{Chebyshev Polynomials}.
Chapman \& Hall/CRC, Boca Raton, FL.

\bibitem{Maulana-Wijaya-Santoso2018}
Maulana NR, Wijaya K, Santoso KA. 2018
On chromatic polynomial of a fan graph.
\textit{Majalah Ilmiah Matematika dan Statistika} {\bfseries 18} (2018), 55--60.

\bibitem{Mihalic-etal1992}
Mihali\'{c} Z, Veljan D, Ami\'{c} D,  Nikoli\'{c} S, Plav\v{s}i\'{c} D, Trinajsti\'{c} N. 1992
The distance matrix in chemistry.
\textit{J. Math. Chem.} {\bfseries 11}, 223--258.

\bibitem{Mlotkowski2022}
M\l otkowski W. 2022
Quadratic embedding constants of path graphs.
\textit{Linear Algebra Appl.} {\bfseries 644}, 95--107.

\bibitem{MO-2018}
M{\l}otkowski W, Obata N. 2020
On quadratic embedding constants of star product graphs.
\textit{Hokkaido Math. J.} {\bfseries 49}, 129--163.

\bibitem{Mlotkowski-Obata2025}
M{\l}otkowski W, Obata N. 2025
Quadratic embedding constants of fan graphs and graph joins.
\textit{Linear Algebra Appl.} {\bfseries 709}, 58--91.

\bibitem{MSW2024}
M{\l}otkowski W, Skrzypczyk M, Wojtylak M. 2024
On quadratic embeddability of bipartite graphs and theta graphs.
\textit{arXiv:2409.17662}

\bibitem{Obata2017}
Obata N. 2017
Quadratic embedding constants of wheel graphs.
\textit{Interdiscip. Inform. Sci.} {\bfseries 23}, 171--174.

\bibitem{Obata2023a}
Obata N. 2023
Primary non-QE graphs on six vertices.
\textit{Interdiscip. Inform. Sci.} {\bfseries 29}, 141--156.

\bibitem{Obata2023b}
Obata N. 2023
Complete multipartite graphs of non-QE class.
\textit{Electronic J. Graph Theory Appl.} {\bfseries 11}, 511--527.

\bibitem{Obata-Zakiyyah2018}
Obata N, Zakiyyah AY. 2018
Distance matrices and quadratic embedding of graphs.
\textit{Electronic J. Graph Theory Appl.} {\bfseries 6}, 37--60.

\bibitem{Rayes2005}
Rayes MO, Trevisan V, Wang PS. 2005
Factorization properties of Chebyshev polynomials.
\textit{Comput. Math. Appl.} \textbf{50}, 1231--1240.

\bibitem{Rivlin1974}
Rivlin TJ. 1974
\textit{The Chebyshev Polynomials}.
John Wiley \& Sons, New York.

\bibitem{Roy2017}
Roy S. 2017
Packing chromatic number of certain fan and wheel related graphs.
\textit{AKCE Int. J. Graphs Comb.} {\bfseries  14}, 63--69.


\bibitem{Schoenberg1935}
Schoenberg IJ. 1935
Remarks to Maurice Fr\'echet's article
``Sur la d\'efinition axiomatique d'une classe d'espace 
distanci\'es vectoriellement applicable sur l'espace de Hilbert''.
\textit{Ann. of Math.} {\bfseries 36}, 724--732.

\bibitem{Schoenberg1938}
Schoenberg IJ. 1938
Metric spaces and positive definite functions.
\textit{Trans. Amer. Math. Soc.} {\bfseries 44} (1938), 522--536.

\bibitem{Wolfram2022}
Wolfram DA. 2022
Factoring variants of Chebyshev 
polynomials with minimal polynomials of $\cos(2\pi/d)$.
\textit{Bull. Aust. Math. Soc.} {\bfseries 106}, 448--457.

\bibitem{Yang-Wang-etal2019}
Yang Y, Wang A, Wang H, Zhao WT, Sun DQ. 2019
On subtrees of fan graphs, wheel graphs, and
``partitions'' of wheel graphs under dynamic evolution.
\textit{Mathematics} {\bfseries 2019}, 7(5), 472.


\end{thebibliography}
\end{document}